\documentclass[a4paper,reqno,10pt]{amsart}
\usepackage{amsmath,amssymb,amsfonts,graphicx,epsfig,color,url}
\usepackage{latexsym,mathrsfs,cite}

\usepackage{float}
\usepackage{mathtools}
\usepackage{hyperref} 
\usepackage{comment}
\usepackage[title,titletoc,toc]{appendix}

\usepackage{graphicx}
\usepackage{subcaption}

\makeatletter
\@namedef{subjclassname@2020}{%
  \textup{2020} Mathematics Subject Classification}
\makeatother

\numberwithin{equation}{section}

\title[Standing and traveling waves in QHD with viscosity]{Existence of standing and traveling waves in quantum hydrodynamics with viscosity}

\author{Delyan Zhelyazov}
\address[Delyan Zhelyazov]{Departamento de Matem\'aticas y Mec\'anica\\Instituto de Investigaciones en Matem\'aticas Aplicadas y en Sistemas\\Universidad Nacional Aut\'{o}noma de M\'{e}xico\\Circuito Escolar s/n, Ciudad Universitaria, C.P. 04510\\Cd. de M\'{e}xico (Mexico)}\email{delyan.zhelyazov@iimas.unam.mx}

\newtheorem{theorem}{Theorem}[section]
\newtheorem{lemma}[theorem]{Lemma}
\newtheorem{corollary}[theorem]{Corollary}

\theoremstyle{remark}
\newtheorem{remark}[theorem]{Remark}

\begin{document}

\begin{abstract}
We prove existence of standing waves for two quantum hydrodynamics systems with linear and nonlinear viscosity. Moreover, global existence of traveling waves is proved for the former without restrictions on the viscosity and dispersion parameters, thanks to a suitable Lyapunov function. This is an improvement with respect to the global existence result in \cite{LMZ2020}, where it was required that the viscosity is sufficiently strong.
\end{abstract}


\maketitle

\section{Introduction}\label{sec:intr}
In this paper we study two quantum hydrodynamics (QHD) systems for which we investigate existence of traveling and standing waves. In the case of linear viscosity:
\begin{equation}
\label{eq_sys_l_intro}
\begin {cases}
\displaystyle{\rho_t+m_x=0,} & \\
\displaystyle{m_t+\left(\frac{m^2}{\rho}+p(\rho) \right)_x=\epsilon \mu m_{xx}+\epsilon^2 k^2 \rho \left(\frac{(\sqrt{\rho})_{xx}}{\sqrt{\rho}}\right)_x,} & \\
\end{cases}
\end{equation}
we investigate existence of standing and traveling waves, while for nonlinear viscosity:
\begin{equation}
\label{eq_sys_n_intro}
\begin {cases}
\displaystyle{\rho_t+m_x=0,} & \\
\displaystyle{m_t+\left(\frac{m^2}{\rho}+p(\rho) \right)_x=\epsilon \mu \rho \left(\frac{m_{x}}{\rho}\right)_x+\epsilon^2 k^2 \rho \left(\frac{(\sqrt{\rho})_{xx}}{\sqrt{\rho}}\right)_x,} & \\
\end{cases}
\end{equation}
we study existence of standing waves.
Here $\rho\geq0$ is the density, $m=\rho u$ is the momentum, where $u$ stands fo the velocity, and $p(\rho)=\rho^\gamma$ for $\gamma \geq 1$ is the pressure. The positive constants $0<\epsilon\ll1$, $\mu$, and $k>0$ define the viscosity ($\epsilon \mu$) and the dispersive ($\epsilon^2 k^2$) coefficients. The shape of the dispersive term is known as the Bohm potential, while the nonlinear viscosity chosen in \eqref{eq_sys_n_intro} appears in the theory of superfluidity; see, for instance, \cite{Khalatnikov}, page 109. This term describes the interactions between a super fluid and a normal fluid; in addition, it can also be interpreted as describing the interactions of the fluid with a background.

The system \eqref{eq_sys_n_intro} can be rewritten in a conservative form in the $(\rho,u)$ variable (see \cite{Zhelyazov1}), for this let us first define the enthalpy $h(\rho)$ by 
 \begin{equation*}
 h(\rho) = \begin{dcases}
 \ln \rho, & \gamma = 1,\\
 \frac{\gamma}{\gamma-1}\rho^{\gamma-1}, & \gamma>1.
 \end{dcases}
 \end{equation*}
 Then, the system \eqref{eq_sys_n_intro} can be recast as follows:
\begin{equation}
\label{eq_sys2}
\begin {cases}
\displaystyle{\rho_t+(\rho u)_x=0,}&\\
\displaystyle{u_t + \frac{(u^2)_x}{2}+(h(\rho))_x=\epsilon \mu \Big(\frac{(\rho u)_{x}}{\rho}\Big)_x+\epsilon^2 k^2 \Big(\frac{(\sqrt{\rho})_{xx}}{\sqrt{\rho}}\Big{)}_x.}&
\end{cases}
\end{equation}

Traveling waves (or \emph{dispersive shocks}) for the system \eqref{eq_sys_l_intro} are solutions of the form:
\begin{equation}
\rho(t,x) = P\Big{(}\frac{x - s t}{\epsilon}\Big{)},\mbox{ }m(t,x) = J\Big{(}\frac{x - s t}{\epsilon}\Big{)},\label{trave_waves_l}
\end{equation}
where the speed $s \in \mathbb{R}$ of the traveling wave and its limiting end states
\begin{equation*}
\lim_{y\rightarrow \pm \infty}P(y)=P^{\pm}\  \hbox{and}\ \lim_{y\rightarrow \pm \infty}J(y)=J^{\pm}
\end{equation*}
satisfy the Rankine--Hugoniot conditions for the underlying Euler system:
\begin{equation}
\left\{
\begin{array}{ll}
\rho_t+m_x=0,\\
m_t+\Big{(}\frac{m^2}{\rho}+\rho^\gamma\Big{)}_x=0,
\end{array}
\right.\label{euler_system}
\end{equation}
namely
\begin{align}
\label{Rankine_Hugoniot_l1}
J^+-J^- &= s(P^+-P^-),\\
\label{Rankine_Hugoniot_l2}
\Big{(}\frac{J^2}{P}+P^{\gamma}\Big{)}^+-\Big{(}\frac{J^2}{P}+P^{\gamma}\Big{)}^- &= s(J^+-J^-).
\end{align}

On the other hand standing waves for the system \eqref{eq_sys_n_intro} are given by:
\begin{equation}\label{standing_waves_n}
\rho(t,x)=V\Big(\frac{x}{\epsilon}\Big)^2,\mbox{ }u(t,x)=U\Big(\frac{x}{\epsilon}\Big).
\end{equation}
Standing waves correspond to $s=0$. The end states
\begin{equation*}
V^{\pm} = \lim_{y\rightarrow \pm \infty}V(y),\mbox{ }u^{\pm}=\lim_{y\rightarrow \pm \infty}U(y)
\end{equation*}
are assumed to satisfy the Rankine--Hugoniot conditions for the underlying system
\begin{align*}
&\rho_t+(\rho u)_x=0, \\ 
&u_t + \frac{(u^2)_x}{2}+(h(\rho))_x=0,
\end{align*}
which read
\begin{align}
(V^2 u)^+ - (V^2 u)^- &= 0,\label{RHnon1}\\
\Big(\frac{u^2}{2}+h(V^2)\Big)^+-\Big(\frac{u^2}{2}+h(V^2)\Big)^- &= 0.
\label{RHnon2}
\end{align}
The first studies of models with dispersive terms are \cite{Sagdeev,Gurevich}; see also \cite{Gurevich1,Nov,Hoefer}. Moreover, quantum hydrodynamics systems have been considered from a mathematical perspective in \cite{AM1,AM2,AMtf,AMDCDS,AS,Michele1,Michele,DM,DM1,DFM,GLT,BGL-V19}.

The first attempt to analyze the spectral theory of the linearized operator around dispersive shocks  for the $p-$system with real viscosity and linear capillarity has been discussed in \cite{Humpherys}. Moreover, the case of the Euler formulation \eqref{eq_sys_n_intro}, but with linear viscosity, is investigated in \cite{LMZ2020}, and the related Evans function computations are presented in \cite{LMZ2020a}. The spectrum of the linearization of \eqref{eq_sys2} around dispersive shocks is considered in \cite{LZ2021}.


In \cite{LMZ2020} global existence of traveling waves for the system \eqref{eq_sys_l_intro} was proved in the case of sufficiently strong viscosity. We remove this condition thanks to the presence of a Lyapunov function (see \eqref{Lyapunov_function} below) and obtain a proof of global existence of profiles for arbitrary viscosity and dispersion parameters. The result is similar to the global existence theorem in \cite{Zhelyazov1} for the system with nonlinear viscosity \eqref{eq_sys_n_intro}. Using a related function we proved existence of standing waves for \eqref{eq_sys_l_intro} (see Theorem \ref{theorem_sw_l}).

The remaining part of the paper is organized as follows. In Section \ref{sec:eqns_profiles} we introduce the equations solved by the profiles in the cases of linear and nonlinear viscosity. In Section \ref{sec:non_existence} we prove that there exist no heteroclinic connections between the equilibria of the dynamical system solved by the profile in the case of different end states and zero velocity. Then, in Section \ref{sec:exist} we prove existence of homoclinic loops in the case of equal end states and zero velocity, under appropriate conditions on the end states. Finally, Section \ref{sec:global_existence} is devoted to prove global existence of traveling waves for the QHD system with linear viscosity \eqref{eq_sys_l_intro}.
\section{Equations for the profiles}\label{sec:eqns_profiles}
In this section we will present the equations satisfied by the traveling waves \eqref{trave_waves_l} for \eqref{eq_sys_l_intro} and the standing waves \eqref{standing_waves_n} for \eqref{eq_sys_n_intro}.
\subsection{Linear viscosity}We rewrite the Bohm potential in   conservative form
\begin{equation*}
\rho \Big(\frac{(\sqrt{\rho})_{xx}}{\sqrt{\rho}}\Big{)}_x=\frac{1}{2}\Big{(}\rho (\ln \rho)_{xx}\Big{)}_x\mbox{. }
\end{equation*}
After substituting the profiles $P$ and $J$ in the system \eqref{eq_sys_l_intro} and multiplying by $\epsilon$ we obtain
\begin{align}
- s P'+J'&=0,\label{preq1}\\
-s J' + \Big{(}\frac{J^2}{P}+P^\gamma\Big{)}'&=\mu J''+\frac{k^2}{2}(P(\ln P)'')'\mbox{, }
\label{preq2}
\end{align}
where $'$ denotes $d/dy$ and $P=P(y)$, $J=J(y)$.
Integrating equation \eqref{preq1}, we get
\begin{equation*}
J(y)-sP(y)=J^--sP^-.
\end{equation*}
We can also integrate \eqref{preq1} from $y$ to $+\infty$ to get
\begin{equation*}
J(y)-sP(y)=J^+-sP^+.
\end{equation*}
So we obtain 
\begin{equation}
\label{equation_j}
J(y)=sP(y)-A,
\end{equation}
 with
\begin{equation}
\label{exprA}
A=s P^{\pm}-J^{\pm},
\end{equation}
as follows from the Rankine-Hugoniot condition \eqref{Rankine_Hugoniot_l1}.\\
Substituting the expression for $J(y)$ into equation \eqref{preq2} and integrating we get
\begin{eqnarray*}
\int_{-\infty}^y\bigg{(} -s(sP(x)-A)' + \Big{(}\frac{(sP(x)-A)^2}{P(x)}+P(x)^\gamma\Big{)}' \bigg{)}dx \\
 = \int_{-\infty}^y \bigg{(}\mu sP(x)''+\frac{k^2}{2}(P(x)(\ln P(x))'')' \bigg{)}dx.
\end{eqnarray*}
We can also integrate from $y$ to $+\infty$. We obtain the planar ODE
\begin{equation}\label{2Dsys_l}
P''=\frac{2}{k^2} f(P) - \frac{2 s \mu}{k^2} P' + \frac{P'^2}{P},
\end{equation}
where
\begin{equation}
\label{fun_f}
f(P)=P^{\gamma}-(As+B)+\frac{A^2}{P}.
\end{equation}
Here the constant $B$ is given by
\begin{equation}
\label{expr_B}
B=-s J^{\pm}+\Big{(}\frac{J^2}{P}+P^{\gamma}\Big{)}^{\pm}.
\end{equation}

\subsection{Nonlinear viscosity}
Plugging the ansatz \eqref{standing_waves_n} in  \eqref{eq_sys2} we obtain
\begin{align}
&(V^2 U)'=0,\label{profile1_1}\\
&\frac{(U^2)'}{2}+(h(V^2))'=\mu \Big(\frac{(V^2 U)'}{V^2}\Big)'+k^2\Big(\frac{V''}{V}\Big)',\label{profile1_2}
\end{align}
where $V=V(y)$, $U=U(y)$ and $'$ denotes $d/dy$. Using \eqref{profile1_1} we simplify \eqref{profile1_2}. We get
\begin{align}
&(V^2 U)'=0,\label{profile2_1}\\
&\frac{(U^2)'}{2}+(h(V^2))'= k^2\Big(\frac{V''}{V}\Big)'.
\label{profile2_2}
\end{align}
After integration of equation \eqref{profile2_1}  up to $\pm\infty$ we end up with
\begin{equation}
U=-\frac{C_1}{V^2} \label{eq_v},
\end{equation}
where
\begin{equation}\label{eq:C1} 
C_1=-u^{\pm}(V^{\pm})^2.
\end{equation}
Similarly,  integration of equation \eqref{profile2_2} yields to
\begin{equation}
\frac{U^2}{2}+h(V^2)=k^2\frac{V''}{V}-C_2 \label{eq2D0},
\end{equation}
where
\begin{equation}\label{eq:C2}
C_2=-\frac{(u^{\pm})^2}{2}-h((V^{\pm})^2).
\end{equation}
In view of  \eqref{eq_v} we can eliminate the variable $U$ in \eqref{eq2D0} to obtain the second order equation
\begin{equation}
V''=\frac{g(V)}{k^2},\label{eq_second_order}
\end{equation}
where
\begin{equation*}
g(V)=\Big(\frac{1}{2}\frac{C_1^2}{V^4}+h(V^2)+C_2\Big)V.
\end{equation*}
For standing waves the viscosity term vanishes. Using $V(y)$ we obtain $U(y)$ from
\begin{equation}
U(y) = -C_1{V(y)^2}. \label{eq:U}
\end{equation}
Equation \eqref{eq_second_order} is related to the Schr\"odinger equation for the corresponding wavefunction, which is defined starting form the hydrodynamic variables $\rho$ and $u$ after one introduces the phase $\phi$ as $u = \phi_x$ (see \cite{Gasser}).
\section{Non existence of heteroclinic connections in the case of zero velocity}
In this section we are going to show that equation \eqref{2Dsys_l} does not admit heteroclinic trajectories connecting $[P^-,0]$ to $[P^+,0]$ in the case $s=0$ and $P^+ \neq P^-$. Moreover, we will show that equation \eqref{eq_second_order} does not admit heteroclinic connections between $[V^-,0]$ and $[V^+,0]$ in the case $V^+ \neq V^-$. If $P^+=P^-$ ($V^+=V^-$ in the case of nonlinear viscosity) it follows from the Rankine-Hugoniot conditions that the end states are equal.
\label{sec:non_existence}
\subsection{Nonlinear viscosity}
In this section we consider the case $V^+ \neq V^-$ and $s=0$. We will show that the dynamical system \eqref{eq_second_order}, solved by the profile, does not admit heteroclinic trajectories connecting $[V^-,0]$ to $[V^+,0]$. We introduce the variable $V' = W$ and rewrite \eqref{eq_second_order} as follows:
\begin{align}
V' &= W,\label{ODE_n_1}\\
W' &= \frac{g(V)}{k^2}\label{ODE_n_2}.
\end{align}
The constants $C_1$ and $C_2$ in $g(V)$ can be expressed only in terms of $V^{\pm}$:
\begin{align*}
&g(V)\\
&=\Big(\frac{(V^+V^-)^4}{V^4}\frac{h((V^+)^2)-h((V^-)^2)}{(V^+)^4-(V^-)^4}+h(V^2)-\frac{(V^+)^4 h((V^+)^2)-(V^-)^4 h((V^-)^2)}{(V^+)^4-(V^-)^4}\Big)V.
\end{align*}
The end states $V^{\pm}$ are the positive roots of $g(V) = 0$, as can be verified by using the relations \eqref{eq:C1} and \eqref{eq:C2} in the expression defining $g$. Moreover, a direct calculation shows
\begin{equation*}
g'(V) = \begin{dcases}
-\frac{3 C_1^2 }{2 V^4}+\ln(V^2)+2+C_2, & \gamma = 1\\
- \frac{3 C_1^2 }{2 V^4}+\frac{\gamma(2 \gamma-1)}{\gamma-1} V^{2(\gamma-1)}+C_2, & \gamma >1;
\end{dcases}
\end{equation*}
\begin{equation*}
g''(V) = \begin{dcases}
\frac{6C_1^2}{V^5} + \frac{2}{V}, & \gamma = 1\\
\frac{6C_1^2}{V^5} + 2\gamma(2\gamma-1)V^{2\gamma-3}, & \gamma >1.
\end{dcases}
\end{equation*}
In particular, for $V > 0$ we have $g''(V) > 0$ and therefore $V^{\pm}$ are the only two positive zeros of $g(V)$.
Let
\begin{align*}
G(V): = & \frac{1}{k^2}\int^V g(z) dz \\
= & 
\begin{dcases}
\frac{1}{k^2}\Big(-\frac{C_1^2}{4 V^2}+\frac{1}{2}(C_2-1)V^2+\frac{1}{2}V^2 \ln(V^2)\Big), & \gamma = 1\\
\frac{1}{4 k^2}\Big( -\frac{C_1^2}{V^2}+2 C_2 V^2 +\frac{2 }{\gamma - 1}V^{2 \gamma} \Big), & \gamma >1.
\end{dcases}
\end{align*}
The system \eqref{ODE_n_1}-\eqref{ODE_n_2} has an energy
\begin{equation*}
H_1(V,W) = G(V) - \frac{W^2}{2} - G(V^-),
\end{equation*}
in the case $0<V^+<V^-$, and
\begin{equation*}
\tilde{H}_1(V,W) = G(V) - \frac{W^2}{2} - G(V^+),
\end{equation*}
in the case $0 < V^- < V^+$.

Now we are going to show that one of the equilibria is a strict local maximum of the energy, so there are no trajectories converging to it.
\begin{lemma}
\label{non_existence_non}
Suppose $s=0$, $V^{\pm} > 0$ and $V^+ \neq V^-$. Then, there are no trajectories of system \eqref{ODE_n_1}-\eqref{ODE_n_2}, connecting $[V^-,0]$ to $[V^+,0]$.
\end{lemma}
\begin{proof}
\textit{Case 1: $0 < V^+ < V^-$.} In this case, since $g''(V) > 0$ we have $g'(V^+)<0$ and $g'(V^-)>0$. Moreover,
\begin{align*}
\frac{\partial H_1}{\partial V}(V^+,0) &= 0,\mbox{ }\frac{\partial H_1}{\partial W}(V^+,0) = 0,\\
\frac{\partial^2 H_1}{\partial V^2}(V^+,0) &= \frac{g'(V^+)}{k^2},\mbox{ }\frac{\partial^2 H_1}{\partial V \partial W}(V^+,0) = 0,\mbox{ }\frac{\partial^2 H_1}{\partial W^2}(V^+,0) = -1.
\end{align*}
The Hessian of $H_1(V,W)$, at $V = V^+$ and $W = 0$ is
\begin{equation*}
\begin{bmatrix}
\frac{g'(V^+)}{k^2} & 0\\
0 & -1
\end{bmatrix}.
\end{equation*}
Since $g'(V^+) < 0$, $[V^+,0]$ is a strict local maximum of $H_1(V,W)$. Therefore, there are no trajectories of \eqref{ODE_n_1}-\eqref{ODE_n_2} converging to $[V^+,0]$ as $y \rightarrow +\infty$.

\textit{Case 2: $0 < V^- < V^+$.} In this case, since $g''(V) > 0$ we have $g'(V^+)<0$ and $g'(V^-) > 0$. So, $[V^-,0]$ is a strict local maximum of the energy $\tilde{H}_1(V,W)$. Therefore, there are no trajectories of \eqref{ODE_n_1}-\eqref{ODE_n_2} converging to $[V^-,0]$ as $y \rightarrow -\infty$.
\end{proof}
\subsection{Linear viscosity}
We will consider the case $P^+ \neq P^-$ and $s = 0$. We will show that the dynamical system \eqref{2Dsys_l} does not admit heteroclinic connections between $[P^-,0]$ and $[P^+,0]$.
Setting $s=0$ in \eqref{2Dsys_l}, \eqref{fun_f}, \eqref{exprA}, \eqref{expr_B} we get that the profiles satisfy the ODE:
\begin{equation}\label{ODE_l_s_0}
P''=\frac{2}{k^2} f(P) + \frac{P'^2}{P},
\end{equation}
with
\begin{equation*}
f(P) = P^{\gamma}-B+\frac{A^2}{P},
\end{equation*}
and
\begin{align}
A &=-J^{\pm}, \label{eq:A_s0}\\
B &= \Big{(}\frac{J^2}{P}+P^{\gamma}\Big{)}^{\pm}.\label{eq:B_s0}
\end{align}
We introduce the variable $P'=Q$ and rewrite \eqref{ODE_l_s_0} as follows:
\begin{align}
P' &=Q,\label{sys_first_order_l_1}\\
Q' &= \frac{2}{k^2} f(P)+ \frac{Q^2}{P}.\label{sys_first_order_l_2}
\end{align}
The constants $A$ and $B$ in $f$ can be expressed in terms of $P^{\pm}$:
\begin{align*}
f(P)&= P^\gamma +\frac{P^-P^+}{P}\frac{(P^+)^\gamma-(P^-)^\gamma}{P^+-P^-}\nonumber\\
&\ -\frac{(P^+)^{\gamma+1}-(P^-)^{\gamma+1}}{P^+-P^-}.
\end{align*}
The end states $P^{\pm}$ are the positive roots of $f(P)=0$, as can be verified using the relations \eqref{eq:A_s0}-\eqref{eq:B_s0}. By direct calculation
\begin{align*}
f'(P) &= \gamma P^{\gamma-1}-\frac{P^- P^+}{P^2}\frac{(P^+)^\gamma-(P^-)^\gamma}{P^+-P^-},\\
f''(P) &= \gamma (\gamma-1)P^{\gamma-2}+\frac{2 P^- P^+}{P^3}\frac{(P^+)^\gamma-(P^-)^\gamma}{P^+-P^-}.
\end{align*}
In particular, for $P>0$  we have $f''(P) > 0$ as a sum of a non-negative and a positive term. Therefore, $P^{\pm}$ are the only two positive zeros of $f$. Let
\begin{align*}
F(P): = & \frac{2}{k^2}\int^P \frac{f(z)}{z^2} dz \\
= & 
\begin{dcases}
\frac{2}{k^2}\Big(\ln P + \frac{B}{P} - \frac{A^2}{2 P^2}\Big), & \gamma = 1\\
\frac{2}{k^2}\Big( \frac{P^{\gamma - 1}}{\gamma - 1} + \frac{B}{P} - \frac{A^2}{2 P^2}\Big), & \gamma >1.
\end{dcases}
\end{align*}
The system \eqref{sys_first_order_l_1}-\eqref{sys_first_order_l_2} has a conserved energy
\begin{equation*}
H(P,Q) = F(P) - \frac{1}{2}\Big{(}\frac{Q}{P}\Big{)}^2 - F(P^-),
\end{equation*}
in the case $0 < P^+ < P^-$, and
\begin{equation*}
\tilde{H}(P,Q) = F(P) - \frac{1}{2}\Big{(}\frac{Q}{P}\Big{)}^2 - F(P^+),
\end{equation*}
in the case $0 < P^- < P^+$.

Similarly as Lemma \ref{non_existence_l}, the following lemma is proved by showing that one of the equilibria is a strict local maximum of the energy, hence there are no trajectories converging to it.
\begin{lemma}
\label{non_existence_l}
Suppose $s=0$, $P^{\pm} > 0$ and $P^+ \neq P^-$. Then, there are no trajectories of system \eqref{sys_first_order_l_1}-\eqref{sys_first_order_l_2}, connecting $[P^-,0]$ to $[P^+,0]$.
\end{lemma}
\begin{proof}
\textit{Case 1: $0 < P^+ < P^-$.} In this case, since $f''(P) > 0$ we have $f'(P^+)<0$ and $f'(P^-)>0$. Moreover,
\begin{align*}
\frac{\partial H}{\partial P}(P^+,0) &= 0,\mbox{ }\frac{\partial H}{\partial Q}(P^+,0) = 0,\\
\frac{\partial^2 H}{\partial P^2}(P^+,0) &= \frac{2}{k^2} \frac{f'(P^+)}{(P^+)^2},\mbox{ }\frac{\partial^2 H}{\partial P \partial Q}(P^+,0) = 0,\mbox{ }\frac{\partial^2 H}{\partial Q^2}(P^+,0) = -\frac{1}{(P^+)^2}.
\end{align*}
The Hessian of $H(P,Q)$, at $P = P^+$ and $Q = 0$ is
\begin{equation*}
\begin{bmatrix}
\frac{2}{k^2} \frac{f'(P^+)}{(P^+)^2} & 0\\
0 & -\frac{1}{(P^+)^2}
\end{bmatrix}.
\end{equation*}
Since $f'(P^+) < 0$, $[P^+,0]$ is a strict local maximum of $H(P,Q)$. Therefore, there are no trajectories of \eqref{sys_first_order_l_1}-\eqref{sys_first_order_l_2} converging to $[P^+,0]$ as $y \rightarrow +\infty$.

\textit{Case 2: $0 < P^- < P^+$.} In this case, since $f''(P) > 0$ we have $f'(P^-)<0$ and $f'(P^+) > 0$. So, $[P^-,0]$ is a strict local maximum of the energy $\tilde{H}(P,Q)$. Therefore, there are no trajectories of \eqref{sys_first_order_l_1}-\eqref{sys_first_order_l_2} converging to $[P^-,0]$ as $y \rightarrow -\infty$.
\end{proof}
\section{Existence of standing waves}\label{sec:exist}
In Section \ref{sec:non_existence} we showed that in the case of zero velocity and different end states there are no heteroclinic connections. Hence, here we will consider the case when the end states are equal. We will prove existence of standing waves for the systems with linear and nonlinear viscosity \eqref{eq_sys_l_intro} and \eqref{eq_sys2} under appropriate conditions on $(\rho^+, J^+) = (\rho^-, J^-)$ and $(\rho^+, u^+) = (\rho^-, u^-)$, respectively.
\subsection{Standing waves for QHD with nonlinear viscosity}
By Lemma \ref{non_existence_non} equation \eqref{eq_second_order} does not admit trajectories connecting $[V^-,0]$ to $[V^+,0]$ for $V^+ \neq V^-$. Hence, in this section we will consider the case of
\begin{equation}
\label{cond_sw_non}
V^+ = V^-.
\end{equation}
Conditions \eqref{RHnon1} and \eqref{cond_sw_non} imply $u^+ = u^-$. 
In this section we are going to prove existence of standing waves for the quantum hydrodynamics system \eqref{eq_sys2} under appropriate conditions on the end state. We have $h'(\rho) > 0$ and we will use the notation
 \begin{equation*}
 c_s(\rho) = \sqrt{\rho h'(\rho)} = \begin{dcases}
 1, & \gamma = 1,\\
 \sqrt{\gamma \rho^{\gamma - 1}}, & \gamma>1,
 \end{dcases}
 \end{equation*}
for the sound speed. We rewrite \eqref{eq_second_order} as a first order system:
\begin{align}
V' &= W, \label{cons_sys_1}\\
W' &= \frac{g(V)}{k^2}, \label{cons_sys_2}
\end{align}
where
\begin{equation*}
g(V)=\Big(\frac{1}{2}\frac{C_1^2}{V^4}+h(V^2)+C_2\Big)V,
\end{equation*}
with
\begin{align}
C_1 &=-u^+(V^+)^2,\label{C1_eq}\\
C_2 &=-\frac{(u^+)^2}{2}-h((V^+)^2).\label{C2_eq}
\end{align}
System \eqref{cons_sys_1}-\eqref{cons_sys_2} has a conserved energy
\begin{equation*}
H_1(V,W) = G(V) - \frac{W^2}{2} - G(V^+),
\end{equation*}
where
\begin{align*}
G(V): = & \frac{1}{k^2}\int^V g(z) dz \\
= & 
\begin{dcases}
\frac{1}{k^2}\Big(-\frac{C_1^2}{4 V^2}+\frac{1}{2}(C_2-1)V^2+\frac{1}{2}V^2 \ln(V^2)\Big), & \gamma = 1\\
\frac{1}{4 k^2}\Big( -\frac{C_1^2}{V^2}+2 C_2 V^2 +\frac{2 }{\gamma - 1}V^{2 \gamma} \Big), & \gamma >1.
\end{dcases}
\end{align*} 
Using \eqref{C1_eq} and \eqref{C2_eq} it follows that $g(V^+)=0$ so that $[V^+,0]$ is a stationary point for \eqref{cons_sys_1}-\eqref{cons_sys_2}. Moreover, a direct calculation shows:
\begin{equation} \label{eq:gprime}
g'(V) = \begin{dcases}
-\frac{3 C_1^2 }{2 V^4}+\ln(V^2)+2+C_2, & \gamma = 1\\
- \frac{3 C_1^2 }{2 V^4}+\frac{\gamma(2 \gamma-1)}{\gamma-1} V^{2(\gamma-1)}+C_2, & \gamma >1;
\end{dcases}
\end{equation}
\begin{equation*}
g''(V) = \begin{dcases}
\frac{6C_1^2}{V^5} + \frac{2}{V}, & \gamma = 1\\
\frac{6C_1^2}{V^5} + 2\gamma(2\gamma-1)V^{2\gamma-3}, & \gamma >1.
\end{dcases}
\end{equation*}
In particular, for $V>0$ we have $g''(V) > 0$. The linearization of \eqref{cons_sys_1}-\eqref{cons_sys_2} at $V = V^+$ and $W = 0$ is
\begin{equation*}
J=\begin{bmatrix}
0 & 1\\
\frac{g'(V^+)}{k^2} & 0
\end{bmatrix}.
\end{equation*}
For $V>0$ we have 
 \begin{equation}
G(V)-G(V^+)= \frac{1}{k^2} \int_{V^+}^V g(z) dz.\label{eq:G}
 \end{equation}
Homoclinic trajectories of \eqref{cons_sys_1}-\eqref{cons_sys_2} to $[V^+,0]$ correspond to standing waves for \eqref{eq_sys2}.

In the following theorem we analyze the level sets of the energy $H_1(V,W)$ to describe the parameter values for which there exists a homoclinic loop. For $|u^+| = c_s(\rho^+)$ the equilibrium $[V^+,0]$ is nonhyperbolic, so we study the system in its neighborhood. In the case $|u^+| > c_s(\rho^+)$ the equilibrium is a strict local maximum of the energy, which shows that there are no trajectories converging to it.
\begin{theorem}
\label{standing_waves_nonlinear_viscosity}
There exists a homoclinic loop for \eqref{cons_sys_1}-\eqref{cons_sys_2} to $[V^+,0]$ if and only if $0 < |u^+| < c_s(\rho^+)$.
\end{theorem}
\begin{proof}
First we are going to show that if $u^+ = 0$, then \eqref{cons_sys_1}-\eqref{cons_sys_2} does not admit standing waves. In this case we have $C_1 = 0$ and $C_2 = -h((V^+)^2)$. $V>0$ is a solution of $g(V) = 0$ if and only if it is a solution of $h(V^2) + C_2 = 0$. Since $h(V^2)$ is strictly increasing, the latter equation has a unique positive root $V^+$; moreover, $h(V^2) + C_2 < 0$ for $0 < V < V^+$ and $h(V^2) + C_2 > 0$ for $V>V^+$. Hence, $g(V) = 0$ has a unique positive root $V^+$, $g(V) < 0$ for $0 < V < V^+$ and $g(V) > 0$ for $V > V^+$. From \eqref{eq:G} it follows that for $V \in \mathbb{R}^+ - \{V^+\}$ we have
\begin{equation*}
G(V)-G(V^+) > 0.
\end{equation*}
Therefore $H_1(V^+,0) = 0$ and $H_1(V,0) > 0$ for $V \in \mathbb{R}^+ - \{V^+\}$. Then, for $V \in \mathbb{R}^+ - \{V^+\}$ there are two branches solving $H_1(V,W) = 0$, namely
\begin{equation*}
W=\pm \sqrt{2(G(V)-G(V^+))}.
\end{equation*}
From \eqref{eq:gprime} and the expression for $C_2$ we have
\begin{equation*}
g'(V^+) = \begin{dcases}
2, & \gamma = 1\\
2 \gamma (V^+)^{2(\gamma-1)}, & \gamma >1.
\end{dcases}
\end{equation*}
Therefore, $g'(V^+) > 0$, and the eigenvalues of $J$ are
\begin{equation*}
\lambda_{1,2}=\mp\frac{1}{k}\sqrt{g'(V^+)}.
\end{equation*}
We have $\lambda_1 < 0 < \lambda_2$ and therefore $[V^+,0]$ is a saddle for \eqref{cons_sys_1}-\eqref{cons_sys_2}, with associated eigenvectors given by
\begin{equation*}
v_1=\begin{bmatrix}
-\dfrac{k}{\sqrt{g'(V^+)}} \\ 1
\end{bmatrix},\mbox{ }
v_2=-\begin{bmatrix}
\dfrac{k}{\sqrt{g'(V^+)}} \\ 1
\end{bmatrix}.
\end{equation*}
Suppose that an orbit approaches $[V^+,0]$ as $y \rightarrow +\infty$. Then it has to be tangent to $v_1$ or $-v_1$ at $[V^+,0]$.
Denote
\begin{align*}
\tilde{v}_2 &= - v_2,\\
g_1(V) &= \sqrt{2(G(V)-G(V^+))},\mbox{ }V > 0.
\end{align*}
For $V \in \mathbb{R}^+ - \{V^+\}$ we have
\begin{equation*}
g_1'(V) = \frac{1}{k^2}\frac{g(V)}{\sqrt{2(G(V)-G(V^+))}}.
\end{equation*}
Therefore, $g_1'(V) < 0$ for $0 < V < V^+$ and $g_1'(V) > 0$ for $V > V^+$. Denote further
\begin{align*}
S_2 &= \{(V,-g_1(V)) : 0 < V < V^+ \},\\
\tilde{S}_2 &= \{(V, g_1(V)) : V > V^+ \},
\end{align*}
and let $\Gamma_2$ and $\tilde{\Gamma}_2$ be the separatrices of the saddle $[V^+,0]$ which are tangent to $v_2$ and $\tilde{v}_2$ at $[V^+,0]$. Then, we have $\Gamma_2 \subseteq S_2$ and $\tilde{\Gamma}_2 \subseteq \tilde{S}_2$. Moreover,
\begin{align*}
S_2 &\subseteq \{(V,W) : 0 < V < V^+, W < 0 \},\\
\tilde{S}_2 &\subseteq \{(V,W): V > V^+, W > 0\}.
\end{align*}
Therefore, the system \eqref{cons_sys_1}-\eqref{cons_sys_2} admits no homoclinic trajectories for $u^+ = 0$.

\begin{figure}
\begin{center}
\includegraphics[scale=0.8]{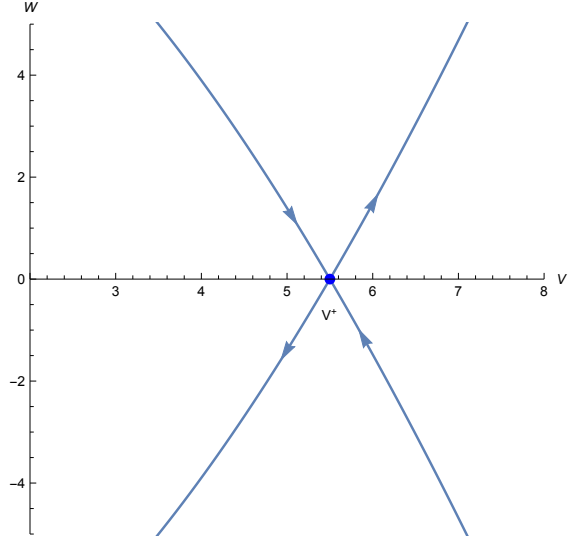}
\end{center}
\caption{The separatrices of the saddle $[V^+,0]$ for parameters $V^+ = 5.5$, $u^+ = 0$, $\gamma = 3/2$, $k = \sqrt{2}$.}
\label{fig_separatrices}
\end{figure}
Suppose $0 < |u^+| < c_s(\rho^+)$. Then, we have $C_1 \neq 0$. It follows from \eqref{eq:gprime} that
\begin{equation*}
g'(V^+) = \begin{dcases}
-2 (u^+)^2 + 2, & \gamma = 1\\
-2 (u^+)^2 +2 \gamma (V^+)^{2(\gamma-1)}, & \gamma >1.
\end{dcases}
\end{equation*}
Using the condition $|u^+| < c_s(\rho^+)$ we get $g'(V^+) > 0$. Moreover, $g(V) \rightarrow +\infty$ as $V \rightarrow 0^+$. Therefore, recalling $g(V^+) = 0$ the equation $g(V)=0$ has a root $V_2$ with $0 < V_2 < V^+$. Since $g''(V) > 0$, $V_2$ and $V^+$ are the only positive zeros of $g(V)$. It follows from \eqref{eq:gprime} that $g'(V) \rightarrow - \infty$ as $V \rightarrow 0^+$ and $g'(V) \rightarrow +\infty$ as $V \rightarrow +\infty$. Since $g''(V) > 0$, then $g'(V)$ is monotonically increasing and has a unique zero $V_0$, which also satisfies $V_2 < V_0 < V^+$. In the interval $V_2\leq V < V^+$ we have 
 \begin{equation*}
 G(V)-G(V^+)= \frac{1}{k^2} \int_{V^+}^V g(z) dz>0.
 \end{equation*}
 Moreover $G'(V)=g(V)/k^2>0$ for $0<V<V_2$ and $G(V) \rightarrow  - \infty$ as $V \to 0^+$.  Hence,  there is a point $V^*\in (0,V_2)$ such that $G(V^*)-G(V^+)=0$. Therefore,   $H_1(V^*,0)=0$,
 $H_1(V,0)<0$ for $0<V<V^*$ and $H_1(V,0)>0$ for $V^*<V<V^+$.
Then,  the system \eqref{cons_sys_1}-\eqref{cons_sys_2} has a  homoclinic loop to $[V^+,0]$, passing through the point $[V^*,0]$, and contained in the  level set $H_1(V,W)=0$.  If we express $W$ as a function of $V$ from $H_1(V,W)=0$, the homoclinic loop can be expressed by the two branches
\begin{equation*}
W=\pm \sqrt{2(G(V)-G(V^-))},
\end{equation*}
for $V^* \leq V \leq V^+$; see Figure \ref{fig_homoclinic}.
\begin{figure}
\begin{center}
\includegraphics[scale=0.8]{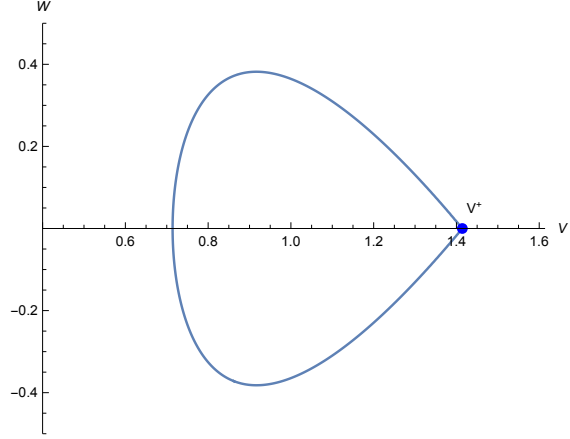}
\end{center}
\caption{The homoclimic loop for parameters $\rho^+ = 2$, $u^+ = 0.8$, $\gamma=3/2$, $k=\sqrt{2}$}
\label{fig_homoclinic}
\end{figure}

Suppose $|u^+| = c_s(\rho^+)$. As before, we have $C_1 \neq 0$. In this case it follows from \eqref{eq:gprime} that $g'(V^+) = 0$. $V^+$ is the unique positive root of the equation $g(V) = 0$, and $g(V) > 0$ for $V \in \mathbb{R}^+ - \{V^+\}$. The linearization of \eqref{cons_sys_1}-\eqref{cons_sys_2} at $V = V^+$ and $W = 0$ is
\begin{equation*}
J=\begin{bmatrix}
0 & 1\\
0 & 0
\end{bmatrix}.
\end{equation*}
$J$ has a double zero eigenvalue. The equilibrium $[V^+,0]$ is nonhyperbolic. Recalling \eqref{eq:G} we obtain
\begin{align*}
&G(V) - G(V^+) < 0,\mbox{ }0 < V < V^+,\\
&G(V) - G(V^+) > 0,\mbox{ }V > V^+.
\end{align*}
Hence, the set $H_1(V,W) = 0$ is given by the two branches
\begin{equation*}
W = \pm \sqrt{2(G(V) - G(V^+))},\mbox{ }V \geq V^+.
\end{equation*}
In order to analyze system \eqref{cons_sys_1}-\eqref{cons_sys_2} in a neighborhood of $[V^+,0]$ we will use \cite{Andreev}.
Let us first change the variables $\xi = W - W^+$, $\eta = W$. Define $\tilde{g}(\xi) = g(\xi + V^+)$ for $\xi > - V^+$. System \eqref{cons_sys_1}-\eqref{cons_sys_2} becomes
\begin{align}
\xi' &= \eta,\label{sys_translated_1}\\
\eta' &= \frac{\tilde{g}(\xi)}{k^2}\label{sys_translated_2}.
\end{align}
The origin is the only equilibrium of \eqref{sys_translated_1}-\eqref{sys_translated_2}.

Let
\begin{equation*}
E = \{ (\xi, \eta): |\xi| < \rho,\mbox{ } |\eta| < \rho \},\mbox{ }\rho \in \mathbb{R}^+.
\end{equation*}
Theorem 6.2.1 on page 128 of \cite{Andreev} applies to systems of the form
\begin{align*}
\xi' &= \eta,\\
\eta' &= f(\xi,\eta),
\end{align*}
where $f$ is a real analytic function on $E$,
\begin{equation*}
f(\xi, \eta) = \bar{f}(\xi) + \bar{g}(\xi)\eta + \bar{h}(\xi,\eta)\eta^2,
\end{equation*}
with
\begin{align*}
\bar{f}(\xi) = a \xi^{\alpha} + a_1 \xi^{\alpha + 1} + \mathcal{O}(\xi^{\alpha + 2}),\mbox{ }a \neq 0, \mbox{ }\alpha \geq 2,\\
\bar{g}(\xi) = b x^{\beta} + b_1 \xi^{\beta + 1} + \mathcal{O}(\xi^{\beta + 2}),\mbox{ }b \neq 0,\mbox{ }\beta \geq 1,
\end{align*}
or $\bar{g}(\xi) = 0$ identically.

We have that $\tilde{g}(\xi)$ is a real analytic function on $E$. Moreover, in our case
\begin{equation*}
\bar{f}(\xi) = \frac{g''(V^+)}{2 k^2}\xi^2 + \mathcal{O}(\xi^3),
\end{equation*}
therefore
\begin{equation*}
a = \frac{\tilde{g}''(0)}{2 k^2} = \frac{g''(V^+)}{2 k^2} > 0,\mbox{ } \alpha = 2,
\end{equation*}
and $\bar{g}(\xi) = 0$ identically. Therefore, case 5) of Theorem 6.2.1, page 128 of \cite{Andreev} applies and the equilibrium 0 is a cusp for \eqref{sys_translated_1}-\eqref{sys_translated_2}. It follows from the dicsussion on page 128 of \cite{Andreev} that there exists a constant $\delta > 0$ and a function
\begin{equation}
\label{fun_traj_large_y}
\phi(V) = \frac{1}{k}\sqrt{\frac{g''(V^+)}{3}}(V-V^+)^{\frac{3}{2}} + o((V-V^+)^{\frac{3}{2}}),\mbox{ }0 < V < \delta,
\end{equation}
such that the following holds: if $[V(y), W(y)]$ is a solution of \eqref{cons_sys_1}-\eqref{cons_sys_2}, defined on interval $I$, which converges to $[V^+,0]$ as $y \rightarrow +\infty$ ($y \rightarrow -\infty$), then there exists $\tilde{y} \in I$ ($\bar{y} \in I$) such that $[V(y), W(y)] \in \hat{S}_2$ ($\hat{S}_1$) for $y > \tilde{y}$ ($y < \bar{y}$), where
\begin{equation*}
\hat{S}_{i} := \{(V,W): 0 < V - V^+ < \delta,\mbox{ }W = \pm \phi(V)\},\mbox{ }i=1,2.
\end{equation*}
It follows from \eqref{fun_traj_large_y} that there exists $\delta_1 \in (0,\delta)$ such that $\phi(V) > 0$ for $0 < V - V^+< \delta_1$.
Let
\begin{equation*}
\hat{g}_1(V) = \sqrt{2(G(V)-G(V^+))},\mbox{ }V > V^+.
\end{equation*}
We have
\begin{equation*}
\hat{g}_1'(V) = \frac{1}{k^2}\frac{g(V)}{\sqrt{2(G(V)-G(V^+))}}.
\end{equation*}
for $V > V^+$. Hence $\hat{g}_1' (V) > 0$ for $V > V^+$. Let
\begin{equation*}
S_{1,2}^* = \{ (V, \pm \hat{g}_1(V)) : V > V^+ \}.
\end{equation*}
We have
\begin{equation*}
S_{1,2}^* \subseteq \{ (V,W) : V > V^+,\mbox{ } W \gtrless 0 \}.
\end{equation*}
The sets $S_1^*$ and $S_2^*$ are disjoint. Moreover, the set $H_1(V,W) = 0$ equals $S_1^* \cup \{[V^+,0]\} \cup S_2^*$.
\begin{figure}
\begin{center}
\includegraphics[scale=0.8]{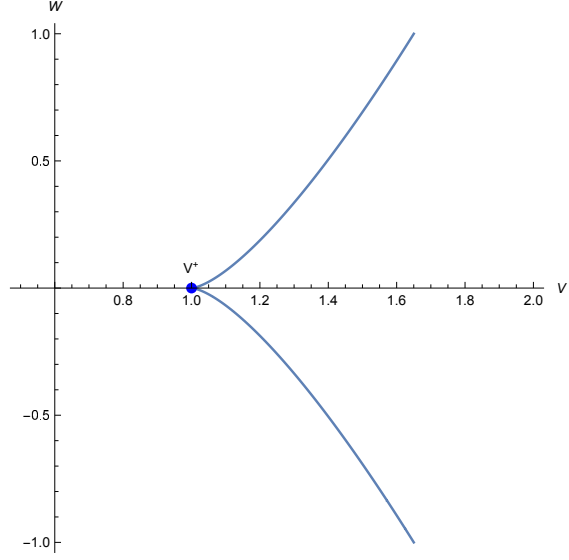}
\end{center}
\caption{The energy level $H_1(V,W) = 0$ for parameters $V^+ = 1$, $u^+ = \sqrt{3/2}$, $\gamma = 3/2$, $k = 1$.}
\label{fig_cusp}
\end{figure}
Suppose $[V(y),W(y)]$ is a non-constant solution of \eqref{cons_sys_1}-\eqref{cons_sys_2}, which converges to $[V^+,0]$ as $y \rightarrow +\infty$ or $y \rightarrow -\infty$, defined on interval $I$. Then, $H_1(V(y),W(y)) = 0$ for $y \in I$. Moreover, let $[V(y),W(y)]$ be a solution of \eqref{cons_sys_1}-\eqref{cons_sys_2} defined on interval $I$. If $[V(y_0),W(y_0)] \in S_1^*$ $(S_2^*)$ for some $y_0 \in I$, then $[V(y),W(y)] \in S_1^*$ ($S_2^*$) for $y \in I$.

Suppose for a contradiction that there exists a non-constant solution $[V_0(y),W_0(y)]$ which converges to $[V^+,0]$ as $y \rightarrow \pm \infty$. Since $[V_0(y),W_0(y)]$ converges to $[V^+,0]$ as $y \rightarrow +\infty$, for $y$ sufficiently large we have $|V_0(y) - V^+| < \delta_1$  and $[V_0(y),W_0(y)] \in \hat{S}_{1}$. Therefore, $\phi(V_1(y)) > 0$. So we have $W_1(y) > 0$. Since $[V_0(y),W_0(y)] \in S_1^* \cup S_2^*$ it follows that $[V_0(y),W_0(y)] \in S_1^*$. Similarly, for $y$ sufficiently large negative we get $[V_0(y),W_0(y)] \in S_2^*$. It follows that $[V_0(y),W_0(y)] \in S_1^* \cap S_2^*$. However $S_1^*$ and $S_2^*$ are disjoint. We get a contradiction. Therefore, there are no homoclinic loops for \eqref{cons_sys_1}-\eqref{cons_sys_2} to $[V^+,0]$.

Suppose $|u^+| > c_s(\rho^+)$. As before, we have $C_1 \neq 0$. It follows from \eqref{eq:gprime} that $g'(V^+)<0$. Moreover, $g(V) \rightarrow +\infty$ as $V \rightarrow +\infty$. So the equation $g(V) = 0$ has a root $V_2$ with $V_2 > V^+$. Since $g''(V) > 0$, $V^+$ and $V_2$ are the only positive roots of $g(V) = 0$. We have
\begin{equation*}
\frac{\partial H_1}{\partial V}(V^+,0) = 0,\mbox{ }\frac{\partial H_1}{\partial W}(V^+,0) = 0.
\end{equation*}
The Hessian of $H_1(V,W)$, at $V = V^+$ and $W = 0$ is
\begin{equation*}
M=\begin{bmatrix}
g'(V^+) & 0\\
0 & -1
\end{bmatrix},
\end{equation*}
hence $[V^+,0]$ is a strict local maximum of $H_1(V,W)$. Therefore, there are no trajectories of \eqref{cons_sys_1}-\eqref{cons_sys_2} converging to $[V^+,0]$ as $y \rightarrow +\infty$.
\end{proof}
The homoclinic trajectory $V$ provides a standing wave. We obtain $U(y)$ from \eqref{eq:U}. For $\gamma = 1$ the condition $|u^+| < c_s(\rho^+)$ reduces to $|u^+| < 1$. The density of the standing wave is lower in the center than for large $|y|$.

\subsection{Standing waves for QHD with linear viscosity} \label{sec:standing_waves_QHD_l}
By Lemma \ref{non_existence_l} equation \eqref{2Dsys_l} does not admit heteroclinic connections between $[P^-,0]$ and $[P^+,0]$ for $s=0$ and $P^+ \neq P^-$. Therefore, in this section we will consider the case
\begin{equation*}
s=0,\mbox{ }P^+ = P^-.
\end{equation*}
The Rankine-Hugoniot condition \eqref{Rankine_Hugoniot_l1} implies $J^+ = J^-$. In this case equation \eqref{2Dsys_l} becomes
\begin{equation}\label{2Dsys_l_s_0}
P''=\frac{2}{k^2} f(P) + \frac{P'^2}{P},
\end{equation}
with
\begin{equation*}
f(P) = P^{\gamma}-B+\frac{A^2}{P},
\end{equation*}
and
\begin{align}
A &=-J^+, \label{eq:A_s_0}\\
B &= \Big{(}\frac{J^2}{P}+P^{\gamma}\Big{)}^+.\label{eq:B_s_0}
\end{align}
As in the case of the system with nonlinear viscosity considered above, the viscosity term vanishes.

Let $\phi \in C^2(\mathbb{R})$ with $\phi(y) >0 $ for $y \in \mathbb{R}$. We have the identity
\begin{equation}
\frac{1}{2}\Big{(}\Big{(}\frac{\phi'}{\phi}\Big{)}^2\Big{)}'=\frac{\phi'}{\phi^2}\Big(\phi''-\frac{(\phi')^2}{\phi}\Big{)}, \label{eq:identity}
\end{equation}
which can be checked by direct calculation. Hence, multiplying \eqref{2Dsys_l_s_0} by $P'/P^2$, we get
\begin{equation*}
\frac{P'}{P^2}\Big(P''-\frac{(P')^2}{P}\Big{)} = \frac{P'}{P^2}\frac{2}{k^2}f(P).
\end{equation*}
Now, using \eqref{eq:identity} we obtain
\begin{equation*}
\frac{1}{2}\Big{(}\Big{(}\frac{P'}{P}\Big{)}^2\Big{)}' = \frac{P'}{P^2}\frac{2}{k^2}f(P).
\end{equation*}
Denote
\begin{align*}
F(P): = & \frac{2}{k^2}\int^P \frac{f(z)}{z^2} dz \\
= & 
\begin{dcases}
\frac{2}{k^2}\Big(\ln P + \frac{B}{P} - \frac{A^2}{2 P^2}\Big), & \gamma = 1\\
\frac{2}{k^2}\Big( \frac{P^{\gamma - 1}}{\gamma - 1} + \frac{B}{P} - \frac{A^2}{2 P^2}\Big), & \gamma >1.
\end{dcases}
\end{align*} 
For $P > 0$ we have
\begin{equation*}
 F(P)-F(P^+)= \frac{2}{k^2} \int_{P^+}^P \frac{f(z)}{z^2} dz.
\end{equation*}
We introduce the variable $P'=Q$ and rewrite \eqref{2Dsys_l_s_0} as a first order system as follows:
\begin{align}
P' &=Q,\label{sys_l_first_order_1}\\
Q' &= \frac{2}{k^2} f(P)+ \frac{Q^2}{P}.\label{sys_l_first_order_2}
\end{align}
Since $P'=Q$, the above calculation suggests the energy
\begin{equation*}
H(P,Q) = F(P) - \frac{1}{2}\Big{(}\frac{Q}{P}\Big{)}^2 - F(P^+).
\end{equation*}
Indeed, let $[P(y),Q(y)]$ be a solution of \eqref{sys_l_first_order_1}-\eqref{sys_l_first_order_2} and $\mathcal{H}(y)=H(P(y),Q(y))$. Then
\begin{equation*}
\mathcal{H}' = \frac{\partial H}{\partial P}P'+\frac{\partial H}{\partial Q}Q'=\frac{\partial H}{\partial P}Q +\frac{\partial H}{\partial Q}\Big{(}\frac{2}{k^2} f(P)+ \frac{Q^2}{P}\Big{)} = 0.
\end{equation*}

Now, we will use $H(P,Q)$ to prove existence of a homoclinic loop, which gives a standing wave.

We have $u^+ = J^+/P^+$. Using \eqref{eq:A_s_0} and \eqref{eq:B_s_0} we obtain $f(P^+) = 0$. Moreover,
\begin{align}
f'(P) &= \gamma P^{\gamma - 1} - \frac{A^2}{P^2}, \label{eq:fprim_s_0}\\
f''(P) &= \gamma(\gamma - 1) P^{\gamma - 2} + \frac{2 A^2}{P^3}. \label{eq:fsec_s_0}
\end{align}
The linearization of \eqref{sys_first_order_l_1}-\eqref{sys_first_order_l_2} at $P = P^+$ and $Q = 0$ is
\begin{equation*}
J=\begin{bmatrix}
0 & 1\\
\frac{2 f'(P^+)}{k^2} & 0
\end{bmatrix}.
\end{equation*}
Similarly as Theorem \ref{standing_waves_nonlinear_viscosity}, in the following theorem we analyze the level sets of $H(P,Q)$ to describe the parameter values for which there exists a homoclinic loop. In the case $|u^+| > c_s(P^+)$ the equilibrium $[P^+,0]$ is a strict local maximum of the energy, which shows that there are no trajectories converging to it.
\begin{theorem}\label{theorem_sw_l}
There exists a homoclinic loop for \eqref{sys_l_first_order_1}-\eqref{sys_l_first_order_2} to $[P^+,0]$ if and only if $0 < |u^+| < c_s(P^+)$.
\end{theorem}
\begin{proof}
Suppose $u^+ = 0$. In this case $A = 0$ and $B = (P^+)^{\gamma}$. Hence, $f'(P) = \gamma P^{\gamma - 1} > 0$ for $P>0$. So, $P^+$ is the only positive zero of $f(P)$; moreover $f(P) > 0$ for $0 < P < P^+$ and $f(P) > 0$ for $P > P^+$. Hence, for $P \in \mathbb{R}^+ - \{P^+\}$ we have
\begin{equation*}
F(P) - F(P^+) > 0.
\end{equation*}
Therefore, $H(P^+,0) = 0$ and $H(P,0) > 0$ for $P \in \mathbb{R}^+ - \{P^+\}$. Then, for $P \in \mathbb{R}^+ - \{ P^+ \}$ there are two branches solving $H(P,Q) = 0$, namely
\begin{equation*}
Q = \pm P \sqrt{2(F(P) - F(P^+))}.
\end{equation*}
Since $A = 0$, from \eqref{eq:fprim_s_0} we get $f'(P^+) = \gamma (P^+)^{\gamma - 1} > 0$. The eigenvalues of $J$ are
\begin{equation*}
\lambda_{1,2} = \mp \frac{1}{k}\sqrt{2 f'(P^+)}.
\end{equation*}
We have $\lambda_1 < 0 < \lambda_2$ and therefore $[P^+,0]$ is a saddle for \eqref{sys_l_first_order_1}-\eqref{sys_l_first_order_2}. The associated eigenvectors are
\begin{equation*}
v_1=\begin{bmatrix}
-\dfrac{k}{\sqrt{2 f'(P^+)}} \\ 1
\end{bmatrix},\mbox{ }
v_2=-\begin{bmatrix}
\dfrac{k}{\sqrt{2 f'(P^+)}} \\ 1
\end{bmatrix}.
\end{equation*}
Suppose that an orbit approaches $[P^+,0]$ as $y \rightarrow +\infty$. Then it has to be tangent to $v_1$ or $-v_1$ at $[P^+,0]$.
Denote
\begin{align*}
\tilde{v}_2 &= - v_2,\\
h(P) &= P\sqrt{2(F(P) - F(P^+))},\mbox{ }P > 0.
\end{align*}
Denote further
\begin{align*}
S_2 &= \{(P,-h(P)) : 0 < P < P^+ \},\\
\tilde{S}_2 &= \{(P, h(P)) : P > P^+ \},
\end{align*}
and let $\Gamma_2$ and $\tilde{\Gamma}_2$ be the separatrices of the saddle $[P^+,0]$ which are tangent to $v_2$ and $\tilde{v}_2$ at $[P^+,0]$. Then, we have $\Gamma_2 \subseteq S_2$ and $\tilde{\Gamma}_2 \subseteq \tilde{S}_2$. Moreover,
\begin{align*}
S_2 &\subseteq \{(P,Q) : 0 < P < P^+, Q < 0 \},\\
\tilde{S}_2 &\subseteq \{(P,Q): P > P^+, Q > 0\}.
\end{align*}
Therefore, the system \eqref{sys_l_first_order_1}-\eqref{sys_l_first_order_2} admits no homoclinic trajectories for $u^+ = 0$.

Suppose 
\begin{equation}
0 < |u^+| < c_s(P^+).\label{eq:assumption_l}
\end{equation}
It holds that $J^+ = P^+ u^+ \neq 0$, so $A \neq 0$. Hence, $f''(P)>0$ for $P>0$. Moreover, $f'(P) \rightarrow -\infty$ as $P \rightarrow 0^+$ and $f'(P) \rightarrow +\infty$ as $P \rightarrow + \infty$.
From \eqref{eq:assumption_l} we get
\begin{equation*}
f'(P^+) = \gamma (P^+)^{\gamma - 1} - (u^+)^2 > 0.
\end{equation*}
Hence, $f(P)$ has a zero $P_2$ which satisfies $0 < P_2 < P^+$. Since $f''(P) > 0$, $P_2$ and $P^+$ are the only positive zeros of $f(P)$. Furthermore, $f'(P)$ is monotonically increasing and has a unique zero $P_0$, which in addition verifies $P_2 < P_0 < P^+$. Moreover, $f(P) > 0$ for $P \in (0,P_2) \cup (P^+, +\infty)$ and $f(P) < 0$ for $P_2 < P < P^+$. In the interval $P_2 < P < P^+$ we have
 \begin{equation*}
 F(P)-F(P^+)= \frac{2}{k^2} \int_{P^+}^P \frac{f(z)}{z^2} dz>0.
 \end{equation*}
 Also,
 \begin{equation*}
 F'(P) = \frac{2}{k^2}\frac{f(P)}{P^2}>0
 \end{equation*}
 for $0 < P < P_2$ and $F(P) \rightarrow -\infty$ as $P \rightarrow 0^+$. Therefore, there exists a point $P^* \in (0,P_2)$ such that $F(P^*) - F(P^+) = 0$. So, $H(P^*,0) = 0$, $H(P,0) < 0$ for $0 < P <P^*$ and $H(P,0) > 0$ for $P^* < P < P_2$. Hence, the system \eqref{sys_l_first_order_1}-\eqref{sys_l_first_order_2} has a homoclinic loop to $[P^+,0]$, passing through the point $[P^*,0]$, and contained in the level set $H(P,Q) = 0$. If we express $Q$ as a function of $P$ from $H(P,Q) = 0$, the homoclinic loop can be expressed by the two branches
\begin{equation*}
Q = \pm P\sqrt{2(F(P)-F(P^+))},
\end{equation*}
for $P^* \leq P \leq P^+$.

Suppose $|u^+| = c_s(P^+)$. Similarly as the proof of Theorem \ref{standing_waves_nonlinear_viscosity}, using Theorem 6.2.1 on page 128 of \cite{Andreev} one can show that there are no homoclinic loops to $[P^+,0]$.

Suppose $|u^+| > c_s(P^+)$. As before, we have $A \neq 0$. It follows from \eqref{eq:fprim_s_0} that $f'(P^+) < 0$. Moreover, $f(P) \rightarrow +\infty$ as $P \rightarrow +\infty$. So the equation $f(P) = 0$ has a root $P_2$ with $P_2 > P^+$. Since $f''(P) > 0$, $P^+$ and $P_2$ are the only positive roots of $f(P) = 0$. We have
\begin{equation*}
\frac{\partial H}{\partial P}(P^+,0) = 0,\mbox{ }\frac{\partial H}{\partial Q}(P^+,0) = 0.
\end{equation*}
The Hessian of $H(P,Q)$, at $P = P^+$ and $Q = 0$ is
\begin{equation*}
M=\begin{bmatrix}
\frac{2}{k^2}\frac{f'(P^+)}{(P^+)^2} & 0\\
0 & -\frac{1}{(P^+)^2}
\end{bmatrix}.
\end{equation*}
Therefore, $[P^+,0]$ is a strict local maximum of $H(P,Q)$. There are no trajectories of \eqref{sys_l_first_order_1}-\eqref{sys_l_first_order_2} converging to $[P^+,0]$ as $y \rightarrow +\infty$.
\end{proof}
The homoclinic trajectory $P$ provides a standing wave. We obtain $J(y)$ from \eqref{equation_j}.
\section{Global existence of traveling waves for QHD with linear viscosity}\label{sec:global_existence}

In this section we shall prove existence of traveling waves for the system \eqref{eq_sys_l_intro} in the case of large shocks without restriction on the viscosity and dispersion parameters, thus improving the global existence Lemma from \cite{LMZ2020} which requires that the viscosity is sufficiently strong. This improvement is due to the presence of a better Lyapunov function (see \eqref{Lyapunov_function} below) than the one used in \cite{LMZ2020}. This new Lyapunov function is suggested by the computation in Section \ref{sec:standing_waves_QHD_l} about standing waves in the $(\rho,m)$ variables.

Let us rewrite \eqref{2Dsys_l} as a first order system:
\begin{align}
P' &= Q, \label{sys_l_1}\\
Q' &= \frac{2}{k^2} f(P) - \frac{2 s \mu}{k^2} Q + \frac{Q^2}{P}. \label{sys_l_2}
\end{align}
Recall the expression for $f$ in \eqref{sys_l_2}:
\begin{equation*}
f(P) = P^{\gamma}-(As+B)+\frac{A^2}{P},
\end{equation*}
with
\begin{align}
A &=s P^{\pm}-J^{\pm}, \label{eq_A}\\
B &=-s J^{\pm}+\Big{(}\frac{J^2}{P}+P^{\gamma}\Big{)}^{\pm}.
\end{align}
The constants $A,B$ in $f(P)$ can be expressed in terms of $P^{\pm}$ as follows:
\begin{align}
\label{f_roots0}
f(P)&= P^\gamma +\frac{P^-P^+}{P}\frac{(P^+)^\gamma-(P^-)^\gamma}{P^+-P^-}\nonumber\\
&\ -\frac{(P^+)^{\gamma+1}-(P^-)^{\gamma+1}}{P^+-P^-}.
\end{align}

The proof of existence of a heteroclinic orbit  for   \eqref{sys_l_1}-\eqref{sys_l_2} between these equilibria
is obtained separately in the two cases $s>0$ and $s<0$, and under appropriate conditions for the end states $P^\pm$. 
The latter will be then interpreted afterwards  in terms of super-- and sub--sonicity properties for the corresponding end states $(\rho^\pm, J^\pm,s)$ defining a Lax shock 
for the $\epsilon =0$ reduced system \eqref{euler_system}.
The result will be obtained by showing the existence of  a Lyapunov function for that system  and then via an   application of  the LaSalle invariance principle.
For this, a crucial role will be played by the following  reduced system
\begin{align}
P' &= Q, \label{sysr_l_1}\\
Q' &= \frac{2}{k^2} f(P) + \frac{Q^2}{P}. \label{sysr_l_2}
\end{align}
Following the calculations of Section \ref{sec:standing_waves_QHD_l}, we see that system \eqref{sysr_l_1}-\eqref{sysr_l_2} has a conserved energy
\begin{equation*}
H(P,Q) = F(P) - \frac{1}{2}\Big{(}\frac{Q}{P}\Big{)}^2 - F(P^-),
\end{equation*}
where
\begin{align*}
F(P): = & \frac{2}{k^2}\int^P \frac{f(z)}{z^2} dz \\
= & 
\begin{dcases}
\frac{2}{k^2}\Big(\ln P + \frac{A s + B}{P} - \frac{A^2}{2 P^2}\Big), & \gamma = 1\\
\frac{2}{k^2}\Big( \frac{P^{\gamma - 1}}{\gamma - 1} + \frac{A s + B}{P} - \frac{A^2}{2 P^2}\Big), & \gamma >1.
\end{dcases}
\end{align*}
We have
\begin{align}
f'(P) &= \gamma P^{\gamma-1}-\frac{P^- P^+}{P^2}\frac{(P^+)^\gamma-(P^-)^\gamma}{P^+-P^-}, \label{fprim_l}\\
f''(P) &= \gamma (\gamma-1)P^{\gamma-2}+\frac{2 P^- P^+}{P^3}\frac{(P^+)^\gamma-(P^-)^\gamma}{P^+-P^-}. \label{fsec_l}
\end{align}
Using \eqref{f_roots0} we get that the end states $P^\pm$ are positive roots of $f(P) = 0$. From \eqref{fsec_l} it follows that $f''(P) > 0$ for $P > 0$, and hence $P^\pm$ are the only positive zeros of $f$.

In particular, we will show that there exists a homoclinic loop for \eqref{sysr_l_1}-\eqref{sysr_l_2}, which confines the heteroclinic orbit we are looking for; see Figure \ref{fig_heteroclinic}.
\begin{theorem}
\label{global_existence_linear_viscosity}
Suppose that the end states $P^{\pm}$, $J^{\pm}$ and the speed $s$ satisfy the Rankine-Hugoniot conditions \eqref{Rankine_Hugoniot_l1}-\eqref{Rankine_Hugoniot_l2}.
\begin{enumerate}  
\item[\textit{(i)}]
If $s>0$ and $0<P^+<P^-$, then there exists a heteroclinic trajectory for \eqref{sys_l_1}-\eqref{sys_l_2}, connecting $[P^-,0]$ to $[P^+,0]$. If in addition
\begin{equation*}
\frac{s \mu}{k}<\sqrt{-2 f'(P^+)},
\end{equation*}
then the heteroclinic trajectory is non-monotone.
\item[\textit{(ii)}]
If $s<0$ and $0<P^-<P^+$, then there exists a heteroclinic trajectory for \eqref{sys_l_1}-\eqref{sys_l_2}, connecting $[P^-,0]$ to $[P^+,0]$. If in addition
\begin{equation*}
-\frac{s \mu}{k}<\sqrt{-2 f'(P^-)},
\end{equation*}
then the heteroclinic trajectory is non-monotone.
\end{enumerate}
\end{theorem}
\begin{proof}
\emph{Case (i).}
First we are going to show that $A \neq 0$. Suppose by contradiction $A = 0$. Then from \eqref{eq_A} we get $J^{\pm} = s P^{\pm}$.
Multiplying by $s$ we obtain $s J^{\pm} = s^2 P^{\pm}$, hence
\begin{equation}
s(J^+ - J^-) = s^2(P^+ - P^-).\label{eq:J_l}
\end{equation}
From the Rankine-Hogoniot condition \eqref{Rankine_Hugoniot_l2} we obtain
\begin{equation}
s^2(P^+ - P^-) + (P^+)^{\gamma} - (P^-)^{\gamma} = s (J^+ - J^-)\label{eq:RH_l_2}.
\end{equation}
Substituting the expression for $s(J^+ - J^-)$ from \eqref{eq:J_l} in the right hand side of \eqref{eq:RH_l_2} we obtain
$(P^+)^{\gamma} - (P^-)^{\gamma} = 0$. The strict monotonicity of $\rho^{\gamma}$ on $\rho > 0$ implies $P^+ = P^-$. So, we get a contradiction. Hence, $A \neq 0$.

From \eqref{fprim_l} it follows that $f'(P) \rightarrow -\infty$ as $P \rightarrow 0^+$ and $f'(P) \rightarrow +\infty$ as $P \rightarrow +\infty$ and from \eqref{fsec_l} we have $f''(P)>0$. Hence, $f'(P)$ is monotonically increasing and it has a unique zero $P_0$, which also satisfies $P^+ < P_0 < P^-$. In the interval $P^+\leq P<P^-$ we have 
\begin{equation*}
F(P)-F(P^-)= \frac{2}{k^2} \int_{P^-}^P \frac{f(z)}{z^2} dz>0.
\end{equation*}
Also,
\begin{equation*}
F'(P) = \frac{2}{k^2}\frac{f(P)}{P^2} > 0,
\end{equation*}
for $0 < P < P^+$, and since $A \neq 0$ we get $F(P) \rightarrow -\infty$ as $P \rightarrow 0^+$. So, there is a point $P^* \in (0,P^+)$ such that $F(P^*) - F(P^-) = 0$. Therefore, $H(P^*,0) = 0$, $H(P,0) < 0$ for $0 < P < P^*$ and $H(P,0) > 0$ for $P^* < P <P^+$. Then, reduced system \eqref{sysr_l_1}-\eqref{sysr_l_2} has a homoclinic loop to $[P^-,0]$, passing through the point $[P^*,0]$, and contained in the level set $H(P,Q)=0$. If we express $Q$ as a function of $P$ from $H(P,Q) = 0$, the homoclinic loop can be expressed by the two branches
\begin{equation*}
Q = \pm P \sqrt{2(F(P) - F(P^-))},
\end{equation*}
for $P^* \leq P \leq P^-$; see Figure \ref{fig_heteroclinic}.

We want to prove that the homoclinic loop of \eqref{sysr_l_1}-\eqref{sysr_l_2} defines a confining set for \eqref{sys_l_1}-\eqref{sys_l_2}. Indeed, first of all we see that $H(P,0) > 0$ for any $P \in (P^*,P^-)$ and therefore $H(P,Q)>0$ in the interior of the homoclinic loop. Moreover, let us  consider  a trajectory $[P(y),Q(y)]$  solution of \eqref{sys_l_1}-\eqref{sys_l_2} and let us define $\mathcal{H}(y):=H(P(y),Q(y))$. We have
\begin{equation*}
\mathcal{H}' = \frac{\partial H}{\partial P}P'+\frac{\partial H}{\partial Q}Q'=\frac{2 s \mu}{k^2}\Big{(}\frac{Q}{P}\Big{)}^2 \geq 0.
\end{equation*}
Since $\mathcal{H}' \geq 0$ for all points of the homoclinic loop, we conclude that a trajectory which is inside it at $y = \bar{y}$ will stay inside for all $y \geq \bar{y}$.
\begin{figure}
\begin{center}
\includegraphics[scale=0.8]{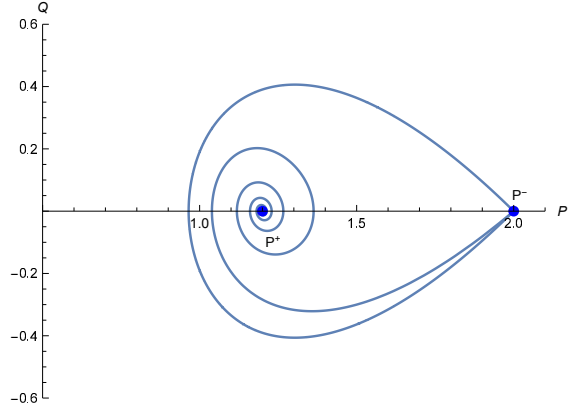}
\end{center}
\caption{The homoclinic loop and the heteroclinic connection for parameters $P^+ = 1.2$, $P^- = 2$, $s = 1$, $\gamma = 3/2$, $\mu = 0.3$, $k = \sqrt{2}$}
\label{fig_heteroclinic}
\end{figure}

Now we are going to show, that the eigenvector, tangent to unstable subspace of the steady-state $[P^-,0]$ is pointing inside the homoclinic loop. The linearization of \eqref{sys_l_1}-\eqref{sys_l_2} at $P^\pm$ and $Q=0$ is $J$ and the linearization of \eqref{sysr_l_1}-\eqref{sysr_l_2} is $\tilde{J}$, where
\begin{equation*}
J=\begin{bmatrix}
0 & 1\\
\frac{2f'(P^{\pm})}{k^2} & -\frac{2 s \mu}{k^2}
\end{bmatrix},\mbox{ }
\tilde{J}=\begin{bmatrix}
0 & 1\\
\frac{2f'(P^{\pm})}{k^2} & 0
\end{bmatrix}.
\end{equation*}
The eigenvalues of $J$ are
\begin{equation*}
\lambda_{1,2} = \frac{-s \mu \pm \sqrt{2 k^2 f'(P^\pm) + s^2 \mu^2}}{k^2},
\end{equation*}
while the eigenvalues of $\tilde{J}$ are
\begin{equation*}
\tilde{\lambda}_{1,2} = \pm \frac{1}{k}\sqrt{2 f'(P^{\pm})}.
\end{equation*}
At the steady-state $[P^-,0]$, since $f'(P^-) > 0$, we have $\lambda_1 > 0$ and $\lambda_2 < 0$ and hence $[P^-,0]$ is a saddle for \eqref{sys_l_1}-\eqref{sys_l_2}. The eigenvector of $J$ corresponding to $\lambda_1$, which is tangent to the unstable manifold of the saddle, is given by
\begin{equation*}
v_1 = -\begin{bmatrix}
\frac{s \mu + \sqrt{2 k^2 f'(P^-) + s^2 \mu^2}}{2 f'(P^-)}\\
1
\end{bmatrix}.
\end{equation*}
Now consider the linearization of \eqref{sysr_l_1}-\eqref{sysr_l_2} at $[P^-,0]$. For the eigenvalues we have $\tilde{\lambda}_1 > 0$ and $\tilde{\lambda}_2 < 0$. The eigenvector of $\tilde{J}$, corresponding to the unstable eigenvalue $\tilde{\lambda}_1$ is
\begin{equation*}
\tilde{v}_1 = -\begin{bmatrix}
\frac{k}{\sqrt{2 f'(P^-)}}\\
1
\end{bmatrix}.
\end{equation*}
If $\tilde{v}_{1,1} > v_{1,1}$, then the eigenvector points inside the homoclinic loop. Since
\begin{equation*}
2 s^2 \mu^2 + 2 s \mu \sqrt{2 k^2 f'(P^-) + s^2 \mu^2} > 0,
\end{equation*}
we get
\begin{equation}
( s \mu + \sqrt{2 k^2 f'(P^-) + s^2 \mu^2})^2 > 2 k^2 f'(P^-). \label{inequality}
\end{equation}
The inequality $\tilde{v}_{1,1} > v_{1,1}$ follows from \eqref{inequality} by taking a square root and dividing by $2 f'(P^-)$. As a consequence, since the set defined by the homoclinic loop is a confining set for \eqref{sys_l_1}-\eqref{sys_l_2}, we conclude that the orbit exiting from $[P^-,0]$ tangent to the eigenvector $v_1$ is trapped inside it.

Now, let us examine the linearization at the steady-state $[P^+,0]$. Since we have $f'(P^+) < 0$, then either
\begin{equation*}
2 k^2 f'(P^+) + s^2 \mu^2 < 0,
\end{equation*}
or
\begin{equation*}
2 k^2 f'(P^+) + s^2 \mu^2  \geq 0 \mbox{ and }\sqrt{2 k^2 f'(P^+) + s^2 \mu^2} < s \mu.
\end{equation*}
In both cases $\Re \lambda_1, \Re \lambda_2 < 0$, that is the steady-state $[P^+,0]$ is stable.

At this point, we need first to exclude that orbits inside the homoclinic loop converge to $[P^-,0]$. To this end, the eigenvector of $J$ at this point, corresponding to the stable eigenvalue $\lambda_2$, is given by
\begin{equation*}
v_2 = \begin{bmatrix}
\frac{-s \mu + \sqrt{2 k^2 f'(P^-) + s^2 \mu^2}}{2 f'(P^-)}\\
-1
\end{bmatrix},
\end{equation*}
while the eigenvector of $\tilde{J}$, again at $[P^-,0]$ and corresponding to $\tilde{\lambda}_2$, is
\begin{equation*}
\tilde{v}_2 = \begin{bmatrix}
\frac{k}{\sqrt{2 f'(P^-)}}\\
-1
\end{bmatrix}.
\end{equation*}
If $\tilde{v}_{2,1} > v_{2,1}$, then the eigenvector, tangent to the stable manifold of the saddle is pointing towards $[P^-,0]$ from outside the homoclinic loop. Since
\begin{equation*}
(k \sqrt{2 f'(P^+)} + s \mu)^2 > 2 k^2 f'(P^+) + s^2 \mu^2 >0,
\end{equation*}
taking a square root we get
\begin{equation*}
k \sqrt{2 f'(P^-)} > -s \mu + \sqrt{2 k^2 f'(P^-) + s^2 \mu^2}
\end{equation*}
and dividing by $2 f'(P^-)$ we obtain $\tilde{v}_{2,1} > v_{2,1}$. As a consequence, no orbits inside the homoclinic loop can converge to $[P^-,0]$.

Now, let $\Omega$ be the compact set defined by the homoclinic loop and its interior. Clearly, as proved before,  $\Omega$ is an invariant set for 
 \eqref{sys_l_1}-\eqref{sys_l_2}. In $\Omega$, consider the function
\begin{equation}
L(P,Q)=\frac{1}{2}\Big{(}\frac{Q}{P}\Big{)}^2-F(P)+F(P^+)=F(P^+)-F(P^-)-H(P,Q) \label{Lyapunov_function},
\end{equation}
and denote by $\mathcal{L}(y): =L(P(y),Q(y))$ the dynamics of this function along an orbit of the dynamical system  \eqref{sys_l_1}-\eqref{sys_l_2}. 
We have $L(P^+,0)=0$ and $\mathcal{L}'=-\mathcal{H}' \leq 0$ in $\Omega$, that is  $L$ is a Lyapunov function for  \eqref{sys_l_1}-\eqref{sys_l_2}.
Denote by $E$  the set of points in $\Omega$ where $\mathcal{L}'=0$, that is $E=\Omega\cap\{Q=0\}$. Hence,  the LaSalle invariance principle implies that any solution $[P(y),Q(y)]$ in $\Omega$  will converge to the largest invariant subset of $E$, which is the set of two steady-states $[P^+,0]$ and $[P^-,0]$. Finally, since as proven before $[P^-,0]$ can not be reached by orbits inside $\Omega$, we can conclude that the trajectory exiting along the unstable manifold of $[P^-,0]$ will converge to the stable  steady-state $[P^+,0]$, proving the existence of the desired heteroclinic connection.
Finally, if $2 k^2 f(P^+)+s^2 \mu^2<0$,  the eigenvalues of the linearization at this point have nonzero imaginary parts, which proves that  in this case the heteroclinic will be oscillatory in a neighborhood of $[P^+,0]$.

\emph{Case (ii)}.
We are going to consider the reverse parameter $\tilde{y} = -y$ and let $\xi$, $\eta$ correspond to $P$, $-Q$, respectively. Denoting $' = d/d\tilde{y}$, the dynamical system \eqref{sys_l_1}-\eqref{sys_l_2} rewrites as follows
\begin{align}
\xi' &= \eta, \label{sys_l_reserve_1}\\
\eta' &= \frac{2}{k^2} f(\xi) - \frac{2 \tilde{s} \mu}{k^2} \eta + \frac{\eta^2}{\xi}, \label{sys_l__reverse_2}
\end{align}
where $\tilde{s} = -s$, $\tilde{s} > 0$.
Hence, we can apply Case (i) to  \eqref{sys_l_reserve_1}-\eqref{sys_l__reverse_2} to conclude there exists a heteroclinic orbit for that system connecting $[P^+,0]$ to $[P^-,0]$. The latter  corresponds to a heteroclinic orbit connecting $[P^-,0]$ to $[P^+,0]$ in the forward parameter $y$ for \eqref{sys_l_1}-\eqref{sys_l_2} for $s<0$ and $P^-<P^+$ and the proof is complete.
\end{proof}
The heteroclinic connection constructed in Theorem \ref{global_existence_linear_viscosity} provides a traveling wave profile for the QHD system with linear viscosity \eqref{eq_sys_l_intro} by using $J(y)$ from equation \eqref{equation_j}.
\begin{remark}
The constants $A$, $B$ in $F(P)$ can be expressed in terms of $P^\pm$ (cf. equation \eqref{f_roots0}):
\begin{equation*}
F(P) = \begin{dcases}
\frac{2}{k^2}\Big(\ln P + \frac{d_1}{P} - \frac{d_2}{2 P^2}\Big), & \gamma = 1\\
\frac{2}{k^2}\Big( \frac{P^{\gamma - 1}}{\gamma - 1} + \frac{d_1}{P} - \frac{d_2}{2 P^2}\Big), & \gamma >1,
\end{dcases}
\end{equation*}
where
\begin{align*}
d_1 &= \frac{(P^+)^{\gamma+1}-(P^-)^{\gamma+1}}{P^+-P^-},\\
d_2 &= P^-P^+\frac{(P^+)^\gamma-(P^-)^\gamma}{P^+-P^-}.
\end{align*}
\end{remark}
Now, we will state the existence result proved above in terms of the speed $s$ and the end states $[P^{\pm}$, $J^{\pm}]$ of the system \eqref{euler_system}, linked by the Rankine-Hugoniot conditions \eqref{Rankine_Hugoniot_l1}-\eqref{Rankine_Hugoniot_l2}, and in particular in terms of   Lax entropy conditions, as well as super-- or sub--sonicity conditions. 
To this end, let us denote $w=(\rho,m)$. The eigenvalues of the Jacobian of the system \eqref{euler_system} are
\begin{equation*}
\lambda_1(w) = u - c_s(\rho),\mbox{ }\lambda_2(w) = u + c_s(\rho),
\end{equation*}
where, $u = \frac{m}{\rho}$ is the flow velocity and $c_s(\rho) = \sqrt{\gamma \rho^{\gamma - 1}}$ is the sound speed. We recall that a discontinuity $(w^\pm,s)$ verifyng  the Rankine-Hugoniot conditions \eqref{Rankine_Hugoniot_l1}-\eqref{Rankine_Hugoniot_l2} is a Lax $k$--shock, $k=1,2$, if
\begin{equation*}
\lambda_k(w^+)<s<\lambda_k(w^-).
\end{equation*}
Moreover, the state $w^\pm = [P^\pm,J^\pm]$ is referred to as supersonic (resp.\ subsonic) if $|u^\pm| > c_s(P^\pm)$ (resp.\ $|u^\pm| < c_s(P^\pm)$).
\begin{corollary}\label{global_existence_l_entropy}
Suppose the end states $[P^{\pm},J^{\pm}]$ and the speed $s$ satisfy $P^{\pm} > 0$ and $[P^{\pm},J^{\pm};s]$ defines
\begin{enumerate}
\item[\textit{(i)}]
 a Lax 2--shock with a subsonic right state;
\item[\textit{(ii)}]
 a Lax 1--shock with a subsonic left state. 
\end{enumerate}
Then there exists a traveling wave profile connecting $[P^-,J^-]$ to $[P^+,J^+]$.
\end{corollary}
\begin{proof}
First of all, we are going to express $J^{\pm}$ in terms of $P^{\pm}$ from the Rankine-Hugoniot conditions \eqref{Rankine_Hugoniot_l1}-\eqref{Rankine_Hugoniot_l2}. From equation \eqref{Rankine_Hugoniot_l1} we get
\begin{equation}
J^+ = J^- + s(P^+ - P^-) \label{eq_Jp}.
\end{equation}
Substituting $J^+$ in equation \eqref{Rankine_Hugoniot_l2} and dividing by the coefficient of $(J^-)^2$
\begin{equation*}
\frac{P^- - P^+}{P^- P^+} \neq 0,
\end{equation*}
we obtain the quadratic equation
\begin{align}
\label{quad_equation}
&(J^-)^2-2sP^- J^- \nonumber\\
&+\frac{P^-P^+\big{(}(P^+)^\gamma-(P^-)^\gamma-P^-s^2 +((P^-)^2s^2)/(P^+)\big{)}}{P^--P^+}=0.
\end{align}
The two roots of \eqref{quad_equation} are given by $J^{-}_{1,2}= s P^- \pm d$,
where 
\begin{equation*}
d = \sqrt{P^+ P^-}\sqrt{\frac{(P^+)^{\gamma}-(P^-)^{\gamma}}{P^+-P^-}}>0.
\end{equation*}
Substituting these roots in equation \eqref{eq_Jp} yields the two solutions
\begin{equation*}
J^+_{1,2} = s P^+ \pm d.
\end{equation*}

\emph{Case (i).} Since the shock satisfies the Lax condition
\begin{equation*}
\lambda_2(w^+) < s < \lambda_2(w^-)
\end{equation*}
and $w^+$ is subsonic, in particular we have $s > u^+ + c_s(P^+) > 0$. Moreover, since 
\begin{equation*}
s > u^+ + c_s(P^+) > u^+ = \frac{J^+}{P^+},
\end{equation*}
we conclude
\begin{equation*}
J^+ = s P^+ - d
\end{equation*}
and, accordingly, $J^- = s P^- - d$. Using again the Lax condition we get
\begin{equation*}
u^+ - u^- < c_s(P^-) - c_s(P^+),
\end{equation*}
that is
\begin{equation*}
\frac{(P^+ - P^-) d}{P^+ P^-} < c_s(P^-) - c_s(P^+).
\end{equation*}
Since the speed of sound $c_s(P)$ is non decreasing, from the above inequality we conclude $P^+ < P^-$ and we are in Case (i) of Theorem \ref{global_existence_linear_viscosity} for the existence of a profile.

\emph{Case (ii).}
In this case, the shock satisfies the Lax condition
\begin{equation*}
\lambda_1(w^+) < s < \lambda_1(w^-)
\end{equation*}
and, being $w^-$ subsonic, we conclude $s < u^- - c_s(P^-) < 0$. Moreover, since
\begin{equation*}
s < u^- - c_s(P^-) < u^- = \frac{J^-}{P^-},
\end{equation*}
we have $J^- = s P^- + d$ and $J^+ = s P^+ + d$. In addition, from the Lax condition we infer
\begin{equation*}
u^+ - u^- < c_s(P^+) - c_s(P^-),
\end{equation*}
which implies
\begin{equation*}
\frac{(P^- - P^+) d}{P^+ P^-} < c_s(P^+) - c_s(P^-).
\end{equation*}
As before, this inequality implies $P^+ > P^-$ because the sound speed $c_s(P)$ is non decreasing. Finally, we are in Case (ii) of Theorem \ref{global_existence_linear_viscosity} and we can conclude with the existence of a profile.
\end{proof}
\begin{remark}
The conditions of Corollary \ref{global_existence_l_entropy} are only possible \emph{sufficient} conditions which guarantee the validity of the hypotheses of Theorem \ref{global_existence_linear_viscosity}, while other possible regimes for the end states may be considered as well. More precisely, in both cases the subsonic assumptions on the end states are needed solely to determine the sign of the speed $s$ of the traveling wave, which can be obtained in many other cases. For instance, one could replace Case (i) with the case of a Lax 2--shock with a right state with positive velocity  (to have 
$s>u^+ +c_s(P^+)>0$), or  replace Case (ii) with the case of a Lax 1--shock with a left state with negative velocity  (to have 
$s<u^- -c_s(P^-)<0$).
\end{remark}

As said before, once we give end states $P^{\pm}$, the corresponding values for the momentum $J^{\pm}$ are given by
\begin{align}
J^+_{1,2} &= s P^+ \pm d,\label{Jp12}\\
J^-_{1,2} &= s P^- \pm d,\label{Jm12}
\end{align}
where
\begin{equation*}
d = \sqrt{P^+ P^-}\sqrt{\frac{(P^+)^{\gamma} - (P^-)^{\gamma}}{P^+ - P^-}} > 0.
\end{equation*}
For completeness and clarity, let us now analyze the aforementioned two possibilities for the momenta $J^{\pm}$ in terms of the Lax conditions, starting for the case $0 < P^+ < P^-$. More precisely, in that case, we shall prove that the solution $J^{\pm}_1$ should not be considered, because the corresponding shock $(P^{\pm},J^{\pm}_1,s)$ will not be an admissible Lax shock for \eqref{euler_system}, while $J^{\pm}_2$ will define a Lax 2--shock for that system.

To this end, let us suppose $0 < P^+ < P^-$ and assume $J^{\pm}_1$ defines a Lax 1--shock, that is
\begin{equation*}
\frac{J^+_1}{P^+} - c_s(P^+) < s < \frac{J^-_1}{P^-} - c_s(P^-).
\end{equation*}
This implies
\begin{equation*}
\frac{d}{P^+} -\frac{d}{P^-} <  c_s(P^+) - c_s(P^-),
\end{equation*}
which is impossible because
\begin{equation*}
\frac{d}{P^+} - \frac{d}{P^-} > 0
\end{equation*}
and, since $c_s(P)$ is non-decreasing, $c_s(P^+) - c_s(P^-) \leq 0$. If $J^{\pm}_1$ satisfies the condition for Lax 2--shock, then in particular $J^+_1/P^+ + c_s(P^+) < s$. Using again the expression \eqref{Jp12} for $J^+_1$ we end up with
\begin{equation*}
\frac{d}{P^+} + c_s(P^+) < 0,
\end{equation*}
which is impossible because $P^+$, $d$ and $c_s(P^+)$ are strictly positive. Hence, the solution $J^{\pm}_1$ can not be considered because it defines a discontinuity which is not admissible.

Now, let us check the conditions verified by $J^{\pm}_2$, starting by proving it does not satisfy those of Lax 1--shock. Indeed, if we assume by contradiction these conditions are verified, then in particular we have $s < J^-_2/P^- - c_s(P^-)$. Using the expression \eqref{Jm12}, this is equivalent to
\begin{equation*}
0 > \frac{d}{P^-} + c_s(P^-),
\end{equation*}
which is again impossible because $d$, $P^-$, $c_s(P^-) > 0$. Finally, let us check $J_2^{\pm}$ verifies the conditions for a Lax 2--shock, namely
\begin{equation}
\frac{J^+_2}{P^+} + c_s(P^+) < s < \frac{J^-_2}{P^-} + c_s(P^-). \label{Lax2}
\end{equation}
In the following we will use the ratio $r = \frac{P^-}{P^+} > 1$, being $0 < P^+ < P^-$. From \eqref{Jp12}, the relation $J^+_2/P^+ + c_s(P^+) < s$ is equivalent to
\begin{equation*}
\sqrt{\gamma (P^+)^{\gamma-1}} < \sqrt{\frac{P^-}{P^+}}\sqrt{\frac{(P^+)^{\gamma} - (P^-)^{\gamma}}{P^+ - P^-}}.
\end{equation*}
Squaring and multiplying by $P^+(P^+ - P^-) < 0$, we get the equivalent inequality
\begin{equation*}
\gamma (P^+)^{\gamma + 1} - \gamma (P^+)^{\gamma} P^- > (P^+)^{\gamma} P^- - (P^-)^{\gamma + 1}.
\end{equation*}
Dividing by $(P^+)^{\gamma + 1}$, we obtain for $r = \frac{P^-}{P^+}>1$
\begin{equation}
r^{\gamma + 1} - (\gamma + 1)r + \gamma >0.\label{Lax_2s_1}
\end{equation}
With the notation $\tilde{f}(r) = r^{\gamma + 1} - (\gamma + 1)r + \gamma$, we have $\tilde{f}(1) = 0$ and
\begin{equation*}
\tilde{f}(r) = (\gamma + 1) \int_1^r (z^{\gamma} - 1)dz > 0,
\end{equation*}
for $r > 1$, which proves \eqref{Lax_2s_1}.

The second inequality in \eqref{Lax2} is proved similarly. Indeed, using the expression \eqref{Jm12} for $J^-_2$, $s < J^-_2/P^- + c_s(P^-)$ is equivalent to
\begin{equation*}
\sqrt{\frac{P^+}{P^-}}\sqrt{\frac{(P^+)^{\gamma} - (P^-)^{\gamma}}{P^+ - P^-}} < \sqrt{\gamma (P^-)^{\gamma - 1}}.
\end{equation*}
Squaring and multiplying by $P^-(P^+ - P^-) < 0$ yields
\begin{equation*}
(P^+)^{\gamma + 1} - P^+ (P^-)^{\gamma} > \gamma P^+ (P^-)^{\gamma} - \gamma (P^-)^{\gamma + 1}.
\end{equation*}
Dividing by $(P^+)^{\gamma + 1}$, we end up to
\begin{equation}
\gamma r^{\gamma + 1} - (\gamma + 1) r^{\gamma} + 1 > 0, \label{Lax_2s_2}
\end{equation}
for $r = \frac{P^-}{P^+} > 1$. Considering this time $\bar{f}(r) = \gamma r^{\gamma + 1} - (\gamma + 1) r^{\gamma} + 1$ one has $\bar{f}(1) = 0$ and
\begin{equation*}
\bar{f}(r) = \gamma(\gamma + 1)\int_1^r(z^{\gamma} - z^{\gamma-1})dz > 0,
\end{equation*}
for $r > 1$. Hence, \eqref{Lax_2s_2} is satisfied.

In the case $0<P^- < P^+$, similar arguments show that the momentum $J^{\pm}_2$ should not be considered, because the resulting shock $(P^{\pm}, J^{\pm}_2, s)$ would not be admissible for \eqref{euler_system}, while $J^{\pm}_1$ leads to an admissible Lax 1--shock; we leave the details to the reader.

\section*{Acknowledgement}
The main part of this work has been carried out while I was a Post-doc at DISIM, Department of Information Engineering, Computer Science and Mathematics, University of L'Aquila. I thank Corrado Lattanzio for suggestions and rereading of the first draft.

\end{document}